\documentclass[12pt,a4paper,reqno]{amsart}
\usepackage{lineno}
\allowdisplaybreaks
\usepackage{amsmath}
\usepackage{amssymb}
\usepackage{amsfonts}
\usepackage{amsthm,latexsym,enumerate,url,cases}
\numberwithin{equation}{section}
\usepackage{mathrsfs}
\usepackage{array,longtable}

\usepackage{xcolor}
\definecolor{cite}{rgb}{0.00,0.60,1.00}
\definecolor{url}{rgb}{1.00,0.10,0.80}
\definecolor{link}{rgb}{0.00,0.00,1.00}
\usepackage[colorlinks,linkcolor=link,urlcolor=url,citecolor=cite,breaklinks]{hyperref}

\hypersetup{
	pdfstartpage=1,
	pdfstartview=FitH}
\def\N{{\mathbb N}}

\theoremstyle{plain}
\newtheorem{theorem}{Theorem}

\newtheorem{question}{Question}
\newtheorem{corollary}{Corollary}
\newtheorem{proposition}{Proposition}

\theoremstyle{definition}

\usepackage{etoolbox}
\makeatletter
\patchcmd{\@settitle}{\uppercasenonmath\@title}{}{}{}
\patchcmd{\@setauthors}{\MakeUppercase}{}{}{}
\patchcmd{\section}{\scshape}{}{}{}
\makeatother

\begin{document}

\title
[Restricted partition functions and additive complements]
{Restricted partition functions and additive complements}

\author
[Y. Ding] {Yuchen Ding}

\address{(Yuchen Ding$^{1,2}$) $^1$School of Mathematics, Yangzhou University, Yangzhou 225002, People's Republic of China}
\address{$^2$HUN-REN Alfr\'ed R\'enyi Institute of Mathematics, Budapest, Pf. 127, H-1364 Hungary}
\email{ycding@yzu.edu.cn}


\begin{abstract}
Let $\N$ be the set of positive integers. For subsets $\mathcal{A},\mathcal{M}\subseteq \N$ and $n\in \N$, let $p(n,\mathcal{A},\mathcal{M})$ denote the number of representations of $n$ in the form
$$
n=\sum_{a\in \mathcal{A}}m_a a,
$$
where $m_a\in \mathcal{M}\cup \{0\}$ for all $a\in \mathcal{A}$, and only finitely many $m_a$ are nonzero. We prove that there exist two infinite sets $\mathcal{A}=\{a_n\}_{n=1}^{\infty}$ and $\mathcal{M}$ of positive integers such that
$$
\lim_{n\to\infty}\frac{\log a_{n+1}-\log a_n}{\log n}=+\infty,
$$
$p(n,\mathcal{A},\mathcal{M})>0$ for every $n\in\N$, and $p$ has polynomial growth. More generally, we prove a construction that associates restricted partition functions of polynomial growth with additive complements satisfying a simple counting condition. This answers a 2016 question of Dai and Chen in the affirmative.
\end{abstract}
\maketitle

Let $\N$ be the set of positive integers and put $\N_0=\N\cup\{0\}$. For $n\in \N$ and given subsets $\mathcal{A},\mathcal{M}$ of $\N$, Canfield and Wilf \cite{CW12} introduced the restricted integer partition function $p(n,\mathcal{A},\mathcal{M})$, defined as the number of representations of $n$ in the form
$
n=\sum_{a\in \mathcal{A}}m_a a,
$
where $m_a\in \mathcal{M}\cup \{0\}$ for all $a\in \mathcal{A}$, and only finitely many $m_a$ are nonzero. The function $p=p(n,\mathcal{A},\mathcal{M})$ is said to have polynomial growth if there exist a constant $c>0$ and an integer $n_0$ such that $p(n,\mathcal{A},\mathcal{M})\le n^c$ for all $n\ge n_0$.

Answering a question of Canfield and Wilf \cite{CW12} in a stronger form, Alon \cite{Al12} proved that there are two explicit infinite sets $\mathcal{A}$ and $\mathcal{M}$ of positive integers such that $p(n, \mathcal{A}, \mathcal{M})=1$ for all $n\ge 1$. For $\mathcal{A}=\{n!\}_{n\ge 1}$ or for $\mathcal{A}=\{n^n\}_{n\ge 1}$, Alon further proved that there exist $n_0$ and an infinite set $\mathcal{M}$ of positive integers such that 
$$
0<p(n, \mathcal{A}, \mathcal{M})<n^{8+o(1)}
$$ 
for all $n>n_0$, thereby answering two further problems of Ljuji\'c and Nathanson \cite{LN12} in the affirmative. In 2016, via an elementary inductive argument, Dai and Chen \cite{DC16} improved Alon's exponent from $8+o(1)$ to $2+o(1)$. More generally, Dai and Chen obtained a unified result for subsets $\mathcal{A}=\{1=a_1<a_2<\cdots\}$ satisfying
$$
c_1(n+1)^{\theta_1}a_n\le a_{n+1}\le c_2(n+1)^{\theta_2}a_n,\quad n=1,2,3,\ldots,
$$
where $c_2>c_1>0$ and $\theta_2\ge \theta_1>0$ are constants. This led them to the following interesting question \cite[Question 3.5]{DC16}.

\begin{question}[Dai--Chen]\label{question1}
Do there exist two infinite sets $\mathcal{A}=\{a_n\}_{n=1}^{\infty}$ and $\mathcal{M}$ of positive integers such that
$$
\lim_{n\to\infty}\frac{\log a_{n+1}-\log a_n}{\log n}=+\infty,
$$
$p(n,\mathcal{A},\mathcal{M})>0$ for all sufficiently large $n$, and $p$ has polynomial growth?
\end{question}

In this note, we show that the answer to Question \ref{question1} is affirmative. Before proving this corollary, we establish a more general result that relates additive complements to Question \ref{question1}. Let $\mathscr{B}$ and $\mathscr{S}$ be subsets of $\N_0$. We call $\mathscr{B}$ and $\mathscr{S}$ {\it complete additive complements} if every nonnegative integer can be represented as the sum of an element of $\mathscr{B}$ and an element of $\mathscr{S}$, i.e., $\mathscr{B}+\mathscr{S}=\N_0$.
For a set $\mathcal{Y}\subseteq\N_0$, let $\mathcal{Y}(x)$ denote the number of elements of $\mathcal{Y}$ not exceeding $x$. 
Our main result is stated as follows.

\begin{theorem}\label{main-thm}
Let $\mathscr{B}$ and $\mathscr{S}$ be complete additive complements satisfying 
$$
\mathscr{B}(x)\mathscr{S}(x)\ll x.
$$
Suppose that
$$
\mathcal{A}=\big\{2^b: b\in \mathscr{B}\big\} \quad \text{and} \quad \mathcal{M}=\bigg\{\sum_{d\in \mathcal{D}} 2^{d}: \mathcal{D}\subset \mathscr{S},\ 0<|\mathcal{D}|<\infty\bigg\}.
$$
Then $p(n, \mathcal{A}, \mathcal{M})>0$ for every $n\in \N$ and $p$ has polynomial growth.
\end{theorem}
\begin{proof}
By the hypothesis of the theorem we have
\begin{align}\label{eq-1}
\mathscr{B}+\mathscr{S}=\mathbb{N}_0.
\end{align}
We first show that $p(n, \mathcal{A}, \mathcal{M})>0$ for every $n\in \N$. By the binary expansion of $n$, there is a finite nonempty set $\mathcal{J}\subset \N_0$ such that
$
n=\sum_{j\in \mathcal{J}}2^j.
$
By \eqref{eq-1}, each $j\in\mathcal{J}$ can be written as $j=s_j+b_j$, where $s_j\in \mathscr{S}$ and $b_j\in \mathscr{B}$. Thus
\begin{align}\label{eq-2}
n=\sum_{j\in \mathcal{J}}2^{s_j}2^{b_j}.
\end{align}
Combining the terms in \eqref{eq-2} with the same value of $b_j$, we get
\begin{align*}
n=\sum_{b\in\mathscr{B}} \Big(\sum_{j\in\mathcal{J}:\ b_j=b} 2^{s_j}\Big)2^b=
\sum_{a\in \mathcal{A}}m_a a.
\end{align*}
For each fixed $b$, the corresponding $s_j$ are distinct, because the $j$ are distinct. Hence each nonzero inner sum belongs to $\mathcal{M}$. Therefore $m_a\in\mathcal{M}\cup\{0\}$ for every $a\in\mathcal{A}$, and only finitely many $m_a$ are nonzero.

It remains to show that $p$ has polynomial growth. Let $n$ be sufficiently large and put $N=\log n/\log 2$. Take an arbitrary representation counted by $p(n,\mathcal{A},\mathcal{M})$ as
\begin{align}\label{eq-3}
n=\sum_{b\in\mathscr{B}}m_{2^b}2^b.
\end{align}
For each $b$ with $m_{2^b}\ne0$, the uniqueness of binary expansion gives a unique finite nonempty set $\mathcal{D}_b\subset\mathscr{S}$ such that
$$
m_{2^b}=\sum_{s\in\mathcal{D}_b}2^s.
$$
Thus, the representation (\ref{eq-3}) determines a finite set
$$
\mathcal{P}=\big\{(s,b): b\in\mathscr{B},\ s\in\mathcal{D}_b\big\}\subseteq \mathscr{S}\times\mathscr{B}
$$
such that
$$
n=\sum_{(s,b)\in\mathcal{P}}2^{s+b}.
$$
Since all summands are positive, each pair in $\mathcal{P}$ satisfies $s+b\le N$. Hence the map from representations to such sets $\mathcal{P}$ is injective, and so
$$
p(n,\mathcal{A},\mathcal{M})\le 2^{f(N)},
$$
where
$$
f(N)=\#\big\{(s,b)\in\mathscr{S}\times\mathscr{B}:s+b\le N\big\}.
$$
Therefore, we have
\begin{align}\label{eq-4}
f(N)\le \mathscr{S}(N)\mathscr{B}(N)
\le c_0 N=\frac{c_0}{\log 2} \log n,
\end{align}
where $c_0>0$ is a constant. It then follows from \eqref{eq-4} that
$$
p(n,\mathcal{A},\mathcal{M})\le 2^{\frac{c_0}{\log 2}\log n}\le n^{c_0},
$$
completing the proof that $p$ has polynomial growth.
\end{proof}

Now, we turn back to Question \ref{question1}. We shall use the following result of Ruzsa \cite{Ru72} on additive complements of powers of two.

\begin{proposition}[Ruzsa, 1972]\label{lem1}
There is a subset $\mathcal{S}\subset \mathbb{N}$ with $\mathcal{S}(x)\ll x/\log x$ such that every sufficiently large integer is of the form $2^k+s$, where $k\in \mathbb{N}$ and $s\in \mathcal{S}$.
\end{proposition}

\begin{corollary}\label{corollary}
The answer to Question \ref{question1} is affirmative.
\end{corollary}
\begin{proof}
Let
$$
\mathcal{B}=\{0\}\cup\big\{2^k: k=0,1,2,\ldots\big\}.
$$
By Proposition \ref{lem1}, after adding finitely many nonnegative integers to $\mathcal{S}$ if necessary, there exists some $\mathcal{S}\subseteq \N_0$ such that $\mathcal{S}(x)\ll x/\log x$ and
\begin{align*}
\mathcal{B}+\mathcal{S}=\mathbb{N}_0.
\end{align*}
Indeed, finitely many missing values may be put into $\mathcal{S}$, since $0\in\mathcal{B}$, and this does not affect the estimate $\mathcal{S}(x)\ll x/\log x$. Since $\mathcal{B}(x)\ll\log x$ for $x\ge 2$, we have $\mathcal{B}(x)\mathcal{S}(x)\ll x$.

Let
$$
\mathcal{A}=\big\{2^b: b\in \mathcal{B}\big\}:=\{a_n\}_{n=1}^{\infty},
$$
where the elements of $\mathcal{A}$ are listed in increasing order. Then $a_1=1$ and $a_n=2^{2^{n-2}}$ for every $n\ge 2$. Hence
$$
\lim_{n\to\infty}\frac{\log a_{n+1}-\log a_n}{\log n}
=\lim_{n\to\infty}\frac{(2^{n-1}-2^{n-2})\log 2}{\log n}=+\infty.
$$
Put
$$
\mathcal{M}=\bigg\{\sum_{d\in \mathcal{D}} 2^{d}: \mathcal{D}\subset \mathcal{S},\ 0<|\mathcal{D}|<\infty\bigg\}.
$$
Moreover, $\mathcal{S}$ is infinite, otherwise $\mathcal{B}+\mathcal{S}$ would have counting function $O(\log x)$ and could not be all of $\N_0$. Hence $\mathcal{M}$ is infinite. By Theorem \ref{main-thm}, $p(n, \mathcal{A}, \mathcal{M})>0$ for every $n\in\N$ and $p$ has polynomial growth, completing the proof of Corollary \ref{corollary}.
\end{proof}

\section*{Use of generative artificial intelligence}

During the preparation of this work, the author used ChatGPT (OpenAI, GPT-5.5 Thinking) to explore possible approaches to Question \ref{question1}. In particular, ChatGPT suggested considering sets of the form $\mathcal{A}=2^{\mathcal{D}}$ and coefficients represented by binary subset sums, and pointed to a theorem of Ruzsa on additive complements of lacunary sequences. The author independently verified the applicability of the cited results, reorganized the proof and more general theorem, checked all mathematical details and references, and wrote the final manuscript. The author assumes full responsibility for the originality, accuracy, and integrity of the work.

\section*{Declaration of interests}
The author declares that there is no conflict of interest in this research.

\bibliographystyle{plain}

\begin{thebibliography}{abcdefghi}

\bibitem[Al12]{Al12}
N. Alon, Restricted integer partition functions, {\it Integers} {\bf 13} (2013), A16.

\bibitem[CW12]{CW12}
E. R. Canfield and H. S. Wilf, On the growth of restricted integer partition functions, {\it Dev. Math.} {\bf 23} (2012), 39--46.

\bibitem[DC16]{DC16}
L.-X. Dai and Y.-G. Chen, On two problems of Ljuji\'c and Nathanson, {\it C. R. Math. Acad. Sci. Paris} {\bf 354} (2016), 235--238.

\bibitem[LN12]{LN12}
Z. Ljuji\'c and M. Nathanson, On a partition problem of Canfield and Wilf, {\it Integers} {\bf 12A} (2012), A11.


\bibitem[Ru72]{Ru72}
I. Z. Ruzsa, On a problem of P. Erd\H{o}s, {\it Canad. Math. Bull.} {\bf 15} (1972), 309--310.

\end{thebibliography}

\end{document}